\documentclass[11pt, a4paper]{scrartcl}
\usepackage{amsmath, amsthm, amssymb, amsfonts}
\usepackage{graphicx}
\usepackage{subcaption}
\usepackage{url}

\title{Semi-perfect 1-Factorizations of the Hypercube}
\author{Natalie C. Behague 
		\thanks{Supported by an EPSRC doctoral studentship.}
		\\
		\small School of Mathematical Sciences,\\[-0.8ex]
		\small Queen Mary University of London,\\[-0.8ex]
		\small London E1 4NS, UK}

\newtheorem{theorem}{Theorem}

\newtheorem{corollary}[theorem]{Corollary}
\newtheorem{prop}[theorem]{Proposition}
\newtheorem{question}{Question}

\theoremstyle{remark}
\newtheorem*{claim}{Claim}

\newcommand{\GM}{G[\mathcal{M}]}
\DeclareMathOperator{\sign}{sgn}

\begin{document}
\maketitle

\begin{abstract}
A 1-factorization $\mathcal{M} = \{M_1,M_2,\ldots,M_n\}$ of a graph $G$ is called perfect if the union of any pair of 1-factors $M_i, M_j$ with $i \ne j$ is a Hamilton cycle. It is called $k$-semi-perfect if the union of any pair of 1-factors $M_i, M_j$ with $1 \le i \le k$ and $k+1 \le j \le n$ is a Hamilton cycle.

We consider 1-factorizations of the discrete cube $Q_d$. There is no perfect 1-factorization of $Q_d$, but it was previously shown that there is a 1-semi-perfect 1-factorization of $Q_d$ for all $d$. 
 Our main result is to prove that there is a $k$-semi-perfect 1-factorization of $Q_d$ for all $k$ and all $d$, except for one possible exception when $k=3$ and $d=6$.
This is, in some sense, best possible.

We conclude with some questions concerning other generalisations of perfect 1-factorizations.
\end{abstract}



\section{Introduction}

A 1-factorization of a graph $H$ is a partition of the edges of $H$ into disjoint perfect matchings $\{M_1,M_2,\ldots,M_n\}$, also known as 1-factors. 
Let $\mathcal{M} = \{M_1,M_2,\ldots,M_n\}$ be such a 1-factorization.
We say that $\mathcal{M}$ is a perfect factorization if every pair $M_i \cup M_j$ with $i,j$ distinct forms a Hamilton cycle. 
A 1-factorization  $\mathcal{M}$ is called \emph{semi-perfect} if $M_1 \cup M_i$ forms a Hamilton cycle for all $i \ne 1$. 

Kotzig \cite{Kotzig} conjectured that the complete graph $K_{2n}$ has a perfect 1-factorization for all $n \ge 2$. 
This has long been outstanding and has so far only been shown to hold for $n$ prime and $2n - 1$ prime (independently by Anderson and Nakamura \cite{Anderson, Nakamura}), as well as certain other small values of $n$ (see \cite{Wallis} for references).

The existence or non-existence of perfect or semi-perfect 1-factorizations has been studied for various other families of graphs, in particular for the hypercube $Q_d$, for $d \ge 2$.
The hypercube graph $Q_d$ has vertices the subsets of $\{1,2,\ldots,d\}$ and two vertices joined by an edge if they differ in a single element.

We say a vertex of $Q_d$ is even if the set contains an even number of elements, and odd if not. Note that every edge of $Q_d$ goes from an odd vertex to an even vertex and so $Q_d$ is bipartite with one vertex class of odd vertices and one vertex class of even vertices, each of size $2^{d-1}$. 

We say an edge is in direction $i$ if its two endpoints differ in element $i$. This allows us to define some natural 1-factors of $Q_d$, called the \emph{directional matchings}: for each direction $i = 1,\ldots,d$ let $D_i$ be all edges in direction $i$. The collection of all directional matchings is a 1-factorization of $Q_d$, and note that the union of any pair $D_i \cup D_j$, with $i,j$ distinct, is a disjoint union of 4-cycles. Thus any perfect or semi-perfect 1-factorization of $Q_d$ must be  in some sense far from this.

Craft \cite{Craft} conjectured that for every integer $d \ge 2$ there is a semi-perfect 1-factorization of $Q_d$. This was proved independently by Gochev and Gotchev \cite{Gochev} and by Kr\'{a}lovi\v{c} and Kr\'{a}lovi\v{c} \cite{Kralovic} in the case where $d$ is odd, and settled for $d$ even by Chitra and Muthusamy \cite{Chitra}.

Gochev and Gotchev in fact went further and defined $\mathcal{M}$ to be $k$-semi-perfect if $M_i \cup M_j$ forms a Hamilton cycle for every $1\le i \le k$ and $k+1 \le j \le d$. 
They proved that there is a $k$-semi-perfect factorization of $Q_d$ whenever $k$ and $d$ are both even with $k<d$. 

This leads us to wonder how close to a perfect factorization we can get. Is there a k-semi-perfect factorization of $Q_d$ for all $k<d$? Is there a perfect factorization of $Q_d$? If not, what is the maximal number of pairs of 1-factors whose union is a Hamilton cycle? 
Let us introduce some definitions.

For a 1-factorization $\mathcal{M}= \{M_1,M_2,\ldots,M_d\}$ of $Q_d$, we define a graph $G[\mathcal{M}]$
 with vertices labelled 
$M_1, \ldots, M_d$ and an edge between $M_i$ and $M_j$ if $M_i \cup M_j$ is a Hamilton cycle on $H$.
Note that the definitions above can be easily restated using $\GM$: 
$\mathcal{M}$ is perfect if $\GM$ is complete,
 $\mathcal{M}$ is semi-perfect if $\GM$ contains $K_{1,d-1}$ as a subgraph, 
 and $\mathcal{M}$ is $k$ semi-perfect if $\GM$ contains $K_{k,d-k}$ as a subgraph.

With this new notation, we can rephrase our questions and ask what is the maximal number of edges that $\GM$ can contain if $\mathcal{M}$ is a 1-factorization of $Q_d$? Which graphs can $\GM$ be isomorphic to?
It is in fact not possible for $\GM$ to be complete when $d > 2$ (i.e. $\mathcal{M}$ cannot be perfect). More than this, we can show that $\GM$ must be bipartite.

\begin{theorem}[\cite{Laufer}] \label{thm:Bipartite}
Let $H$ be a bipartite graph on two vertex classes each of size $n$, where $n$ is even. 
Let $\mathcal{M}$ be a partition of $H$ into perfect matchings. Then $\GM$ must be bipartite.
\end{theorem}

A version of Theorem \ref{thm:Bipartite} with the weaker conclusion that $\GM$ is not complete has been, according to Bryant, Maenhaut and Wanless \cite{Bryant} proved many times, including by Laufer in 1980 \cite{Laufer}. We re-prove it here for a few reasons, the main one being that we extend the argument slightly to show that $\GM$ is bipartite. 
The proof also introduces ideas that we will be using later (in Theorem \ref{thm:why3hard}). 

In addition, it is hard to find the theorem and its proof in the literature -- in particular, when making the conjecture that there is a semi-perfect 1-factorization of $Q_d$, Craft also asked whether a \emph{perfect} 1-factorization of $Q_d$ could be found. Theorem \ref{thm:Bipartite} is not mentioned in any of the papers that proved Craft's semi-perfect conjecture.

\begin{proof}
Let $X$ and $Y$ be the vertex classes of $H$.
A perfect matching $M$ naturally induces a function ${M: X \rightarrow Y}$, where $(x,M(x))$ is an edge of $M$.

For two perfect matchings $M_i$ and $M_j$, let $\pi_{j,i}$ be the permutation $M_j^{-1}M_i$ on $X$.
Note that $\pi_{i,i} = id$,  $\pi_{i,j} = \pi_{j,i}^{-1}$ and 
$\pi_{k,j}\pi_{j,i} = \pi_{k,i}$.
Note further that if $M_iM_j$ is an edge of $\GM$ then $M_i \cup M_j$ is a Hamilton cycle and so $\pi_{j,i}$ is a cycle of length $n$ on $X$. 

Suppose for a contradiction that $\GM$ contains a odd cycle and let $M_{i_1}$, $M_{i_2}$, $\ldots$, $M_{i_k}$, $M_{i_1}$ be such a cycle. 
The permutations $\pi_{i_2,i_1}, \pi_{i_3,i_2}, \ldots, \pi_{i_k,i_{k-1}}, \pi_{i_1,i_k}$ are all cycles of length $n$. Since $n$ is even, all of these are odd permutations. Now,
\begin{align*} 
1 = \sign(\pi_{i_1,i_1}) &= \sign(\pi_{i_1,i_k}\pi_{i_k,i_{k-1}} \pi_{i_{k-1},i_{k-2}} \ldots \pi_{i_3,i_2} \pi_{i_2,i_1}) \\
&= \sign(\pi_{i_1,i_k})\sign(\pi_{i_k,i_{k-1}})\ldots \sign(\pi_{i_3,i_2})\sign(\pi_{i_2,i_1})\\
&= (-1)^k = -1
\end{align*}
We have a contradiction, hence $\GM$ contains no odd cycles.
\end{proof}

In the light of Theorem \ref{thm:Bipartite}, the only remaining question is whether for any $k,d$ there is a 1-factorization $\mathcal{M}$ of $Q_d$ such that $\GM$ is isomorphic to the complete bipartite graph $K_{k,d-k}$. (Equivalently, whether there is a $k$-semi-perfect 1-factorization of $Q_d$ for every $k$ and $d$, in the language of Gochev and Gotchev.) Section \ref{sec:main} of this paper fully resolves this problem, except for whether $\GM$ can be isomorphic to $K_{3,3}$.

We also explain, in section \ref{sec:direction}, why the $K_{3,3}$ case cannot be resolved with our methods. In particular, the 1-factorizations we construct in the proof of the main theorem have a direction respecting property. We show that any 1-factorization $\mathcal{M}$ of $Q_{6}$ satisfying this direction respecting property cannot have $\GM$ is isomorphic to $K_{3,3}$.

We finish with some open questions. 

\section{Main Theorem}
\label{sec:main}

\begin{theorem} For $k,l \in \mathbb{N}$ not both equal to 3, there is a 1-factorization $\mathcal{M}$ of the hypercube $Q_{k+l}$ such that $\GM$ is isomorphic to the complete bipartite graph $K_{k,l}$. \label{thm:main}
\end{theorem}

To prove the theorem, we will use the following result due to Stong, which concerns the symmetric directed hypercube $\overleftrightarrow{Q_d}$, obtained from $Q_d$ by replacing each edge with two directed edges, one in each direction. 
\begin{theorem}[\cite{Stong}] \label{thm:Stong}
For $d \ne 3$, the symmetric directed hypercube $\overleftrightarrow{Q_d}$ can be partitioned into $d$ directed Hamilton cycles.
\end{theorem}

Stong's result applies to directed cubes, but the following corollary allows us to use it for undirected cubes.

\begin{corollary}\label{cor:2Stong}
For $d \ne 3$, the cube $Q_d$ can be partitioned into 1-factors $A_1, A_2, \ldots, A_d$ and also partitioned into 1-factors $B_1, B_2, \ldots, B_d$ such that $A_i \cup B_i$ is a Hamilton cycle for all $i = 1,2,\ldots,d$.
\end{corollary}
\begin{proof}
Using Theorem \ref{thm:Stong}, partition $\overleftrightarrow{Q_d}$  into directed Hamilton cycles $H_1,H_2,\ldots,H_d$. 
Let $E$ be the even vertices of $\overleftrightarrow{Q_d}$ and $O$ the odd vertices, so that $\overleftrightarrow{Q_d}$ is bipartite with respect to the vertex classes $E$ and $O$. For each $H_i$, we define $A_i$ to be the edges of $H_i$ that go from $E$ to $O$, and $B_i$ to be the edges that go from $O$ to $E$. 

Since $H_1,H_2,\ldots,H_d$ partition $\overleftrightarrow{Q_d}$ , every edge from $E$ to $O$ is in a unique $A_i$ and every edge from $O$ to $E$ is in a unique $B_j$. 
If we now ignore the directions on the edges, every edge of $Q_d$ is in a unique $A_i$ and a unique $B_j$. It is clear that $A_i$ and $B_i$ are perfect matchings and $A_i \cup B_i$ is a Hamilton cycle by construction.
\end{proof}

Note that we have slightly abused notation in the case $d=1$, since $A_1 = B_1 = Q_1$ and so $A_1 \cup B_1$ is a single edge rather than a cycle. This will not matter in the cases $k \ne 3, l=1$, and we will consider the case $k=3,l=1$ separately. 

Corollary \ref{cor:2Stong} together with a theorem of Gochev and Gotchev \cite[Theorem~3.1]{Gochev} is enough to show that it is possible to have $\GM$ isomorphic to $K_{k,n-k}$ for all $k\ne 3$ and all even $n-k$. We will improve on their arguments to deal with all but one of the remaining cases. 

We will split the theorem for three different cases and prove each separately. Before we do so, let us outline the ideas involved. 

We can view the hypercube $Q_{k+l}$ as a $k$-dimensional hypercube whose `vertices' are copies of $Q_l$ (i.e. as the Cartesian product of $Q_k$ and $Q_l$). 
Let us formalise this idea: 
Label the vertices of $Q_k$ as subsets of $\{1,2,\ldots,k\}$ in the usual way. 
For each vertex $u$ of $Q_k$,
we define a different copy of $Q_l$ within $Q_{l+k}$: 
let $Q_l^u$ be the induced subgraph of $Q_{k+l}$ on all vertices $w$ where $w \cap \{1,2,\ldots,k\} = u$. 

Conversely, we can view $Q_{k+l}$ as a $l$-dimensional hypercube whose `vertices' are copies of $Q_k$. 
This time, label the vertices of $Q_l$ as subsets of $\{k+1,k+2,\ldots,k+l\}$ in the natural way. 
For each vertex $v$ of $Q_l$, 
we define a different copy of $Q_k$ within $Q_{l+k}$: 
let $Q_k^v$ be the induced subgraph of $Q_{k+l}$ on all vertices $x$ with $x \cap \{k+1,k+2,\ldots k+l\} = v$.

The most straightforward case of the theorem is when neither $k$ nor $l$ is equal to $3$, proved in Proposition \ref{case:neither3}.
To prove this we use a generalisation of Gochev and Gotchev's construction \cite{Gochev}. 

The idea of the proof is as follows:
first, we construct $k$ disjoint matchings that use only edges in directions $1,\ldots k$. The matchings used within the $Q_k^v$s are those obtained from applying Corollary \ref{cor:2Stong} to $Q_k$.
Next we construct $l$ disjoint matchings that use only edges in directions $k+1,\ldots, k+l$. Similarly, the matchings used within the $Q_l^u$s are those obtained from applying Corollary \ref{cor:2Stong} to $Q_l$.
We then prove that taking the union of a matching of the first kind and a matching of the second kind gives a Hamilton cycle. 

The second case of the theorem is when $k=3$ and $l$ is not equal to $1$ or $3$, proved in Proposition \ref{case:k=3}. We use a similar construction to the first case, the only difference being that while we can use Corollary \ref{cor:2Stong} on $Q_l$, we cannot apply it to $Q_3$. We will instead take directional matchings on the copies of $Q_3$; it turns out this can be made to work here.

Finally, we are left with two cases: $(k,l) = (3,1)$ and $(k,l) =(3,3)$. The first of these is proved in Proposition \ref{case:k=3,l=1} by means of an explicit example. The case $(k,l) =(3,3)$ is left unsolved. The difficulty of these final two cases is discussed in the section \ref{sec:direction}.

The following useful notation is common to the proofs of propositions \ref{case:neither3} and \ref{case:k=3}.
For a perfect matching $M$ and a vertex $v$, we define $M(v)$ to be the other endpoint of the edge containing $v$ in $M$. (Note that this clashes slightly with our notation in Theorem \ref{thm:Bipartite}: by that notation we are here conflating $M$ and $M^{-1}$.) 

\begin{prop} When neither $k$ nor $l$ is equal to 3, there is a 1-factorization $\mathcal{M}$ of the hypercube $Q_{k+l}$ such that $\GM$ is isomorphic to the complete bipartite graph $K_{k,l}$.
\label{case:neither3}
\end{prop}
\begin{proof}
Using Corollary \ref{cor:2Stong} partition $Q_k$ into matchings $A_1,A_2,\ldots,A_k$ and matchings $B_1,B_2,\ldots,B_k$ such that $A_i \cup B_i$ is a Hamilton cycle for all $i$. 

For $i = 1,2,\ldots,k$ define $M_i$ to be the matching on $Q_{k+l}$ defined by taking the following edges:
$$\begin{cases}
A_i & \text{ on } Q_k^{\emptyset} \\
B_i & \text{ on } Q_k^v \text{ for } v\ne \emptyset
\end{cases}
$$
Note that the $M_i$ are all disjoint, and they only use edges in directions $1,2,\ldots,k$.

Also partition $Q_l$ into matchings $X_1,X_2,\ldots,X_l$ and matchings $Y_1,Y_2,\ldots,Y_l$ such that $X_j \cup Y_j$ is a Hamilton cycle for all $j$.

For $j=1,2,\ldots,l$ define $N_j$ to be the matching on $Q_{k+l}$ defined by taking the following edges:
$$\begin{cases}
X_j & \text{ on } Q_l^u \text{ for } u \text{ even} \\
Y_j & \text{ on } Q_l^u \text{ for } u \text{ odd}
\end{cases}
$$
Another way to think of $N_j$ is as containing edges between copies of $Q_k^v$. From an even vertex in $Q_k^v$ we add an edge to the corresponding vertex in $Q_k^{X_j(v)}$, and from an odd vertex in $Q_k^v$ we add an edge to the corresponding vertex in $Q_k^{Y_j(v)}$.

Note that the $N_j$ are all disjoint, and they only use edges in directions $k+1,k+2,\ldots,k+l$. Thus the matchings $\{M_i\}_{i=1}^k \cup \{N_j\}_{j=1}^l$ are all disjoint and form a 1-factorization of $Q_{k+l}$. 

\begin{figure}[h]
    \centering
    \begin{subfigure}[b]{0.45\textwidth}
       \includegraphics[width=\textwidth]{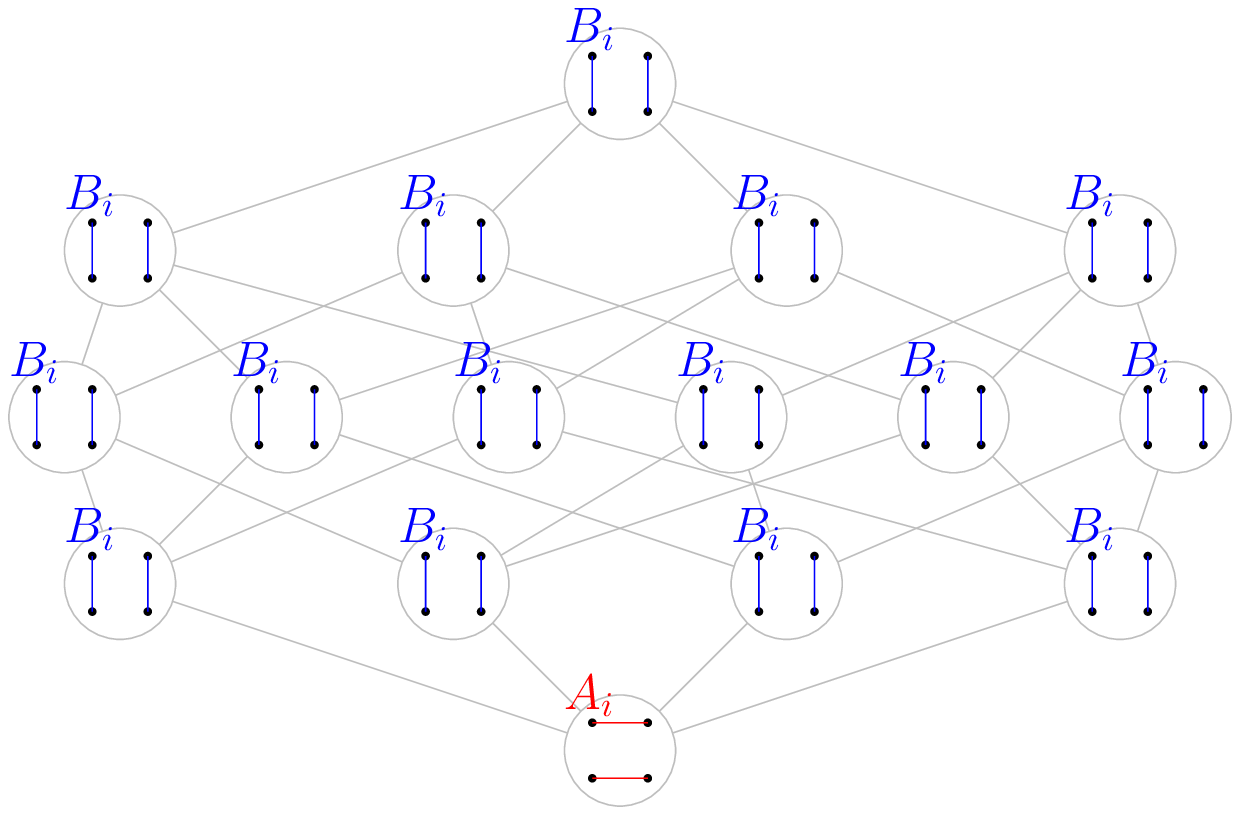}
        \caption{$M_i$}
    \end{subfigure}
    \qquad
    \begin{subfigure}[b]{0.45\textwidth}
        \includegraphics[width=\textwidth]{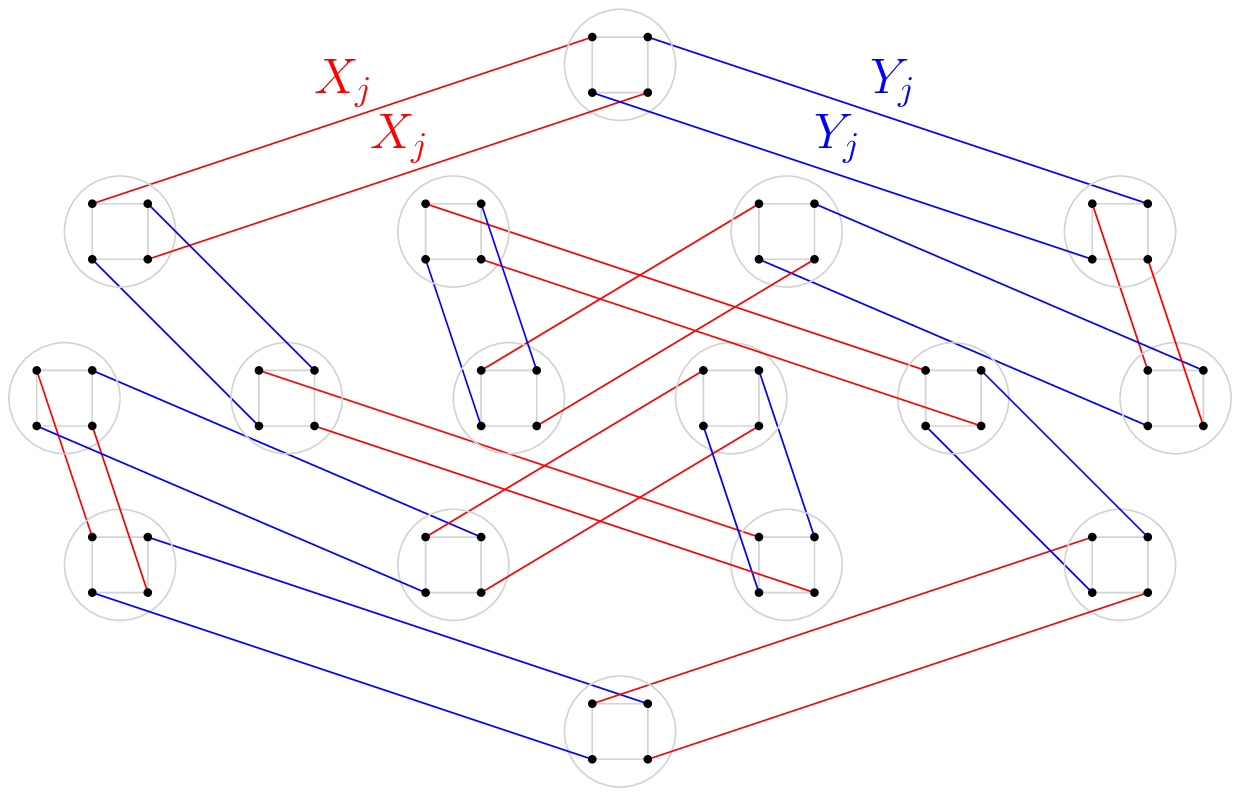}
        \caption{$N_j$}
    \end{subfigure}
    \caption{An example when $k=2$ and $l=4$}\label{fig:general_k,l}
\end{figure}

All that is left is to show that $M_i \cup N_j$ is a Hamilton cycle for all $i,j$. 
Consider following the cycle starting at a vertex $u$ that lies in $Q_k^{\emptyset}$ and alternating between edges first in $N_j$ and then in $M_i$. 

Every time we travel along an edge in $M_i$ the parity of the vertex in $Q_k^v$ switches, and so we will alternate using edges from $X_j$ and edges from $Y_j$ in $N_j$. 
As $X_j \cup Y_j$ is a Hamilton cycle, the first time the cycle returns to $Q_k^{\emptyset}$ we will have travelled through each other $Q_k^v$ exactly once. 

Each time we travel through a different $Q_k^v$ we use an edge from $B_i$ within it. 
After passing through $2^l - 1$ copies of $Q_k^v$ we will have bounced between $u$ and $B_i(u)$ an odd number of times, so the first vertex we encounter in our return to $Q_k^{\emptyset}$ is $B_i(u)$. 
The next vertex would then be $A_i(B_i(u))$.

After passing through $2(2^l)$ distinct vertices (two in each $Q_k^v$) we have moved from $u$ to $A_i(B_i(u))$, i.e. made two steps of the Hamilton cycle $A_i \cup B_i$ within $Q_k^{\emptyset}$.
Thus the first time we will return to $u \cup \emptyset$ is after passing through $2^k2^l$ vertices, which is the total number of vertices in the graph. Hence we have a Hamilton cycle. 
\end{proof}

\begin{prop} For $l$ not equal to $1$ or $3$, there is a 1-factorization $\mathcal{M}$ of the hypercube $Q_{3+l}$ such that $\GM$ is isomorphic to the complete bipartite graph $K_{3,l}$.
\label{case:k=3}
\end{prop}
\begin{proof}
Using Corollary \ref{cor:2Stong}, partition $Q_l$ into matchings $A_1,A_2,\ldots,A_l$ and $B_1,B_2,\ldots, B_l$ such that $A_i\cup B_i$ is a Hamilton cycle for all $j$. 
Let $X_1, X_2$ and $X_3$ be the three directional matchings of $Q_3$ -- that is, $X_j$ contains all edges in direction $j$.

For $i=1,2,\ldots,l$ define $M_i$ to be the matching on $Q_{3+l}$ defined by taking the following edges:
$$\begin{cases}
A_i & \text{ on } Q_l^{\emptyset}, Q_l^{\{1,2\}}, Q_l^{\{1,3\}}, Q_l^{\{2,3\}}  \text{ and } Q_l^{\{1,2,3\}} \\
B_i & \text{ on } Q_l^{\{1\}}, Q_l^{\{2\}} \text{ and } Q_l^{\{3\}}
\end{cases}
$$

For $j = 1,2,3$ define $N_j$ to be the matching on $Q_{3+l}$ defined by taking the following edges, where the subscripts for the $X$s are taken modulo $3$:
$$\begin{cases}
X_{j} & \text{ on } Q_3^v \text{ for $v$ odd} \\
X_{j+1} & \text{ on } Q_3^v \text{ for $v$ even and } v\ne \emptyset \\
X_{j+2} & \text{ on } Q_3^{\emptyset}
\end{cases}
$$

\begin{figure}[h]
    \centering
    \begin{subfigure}[b]{0.28\textwidth}
       \includegraphics[width=\textwidth]{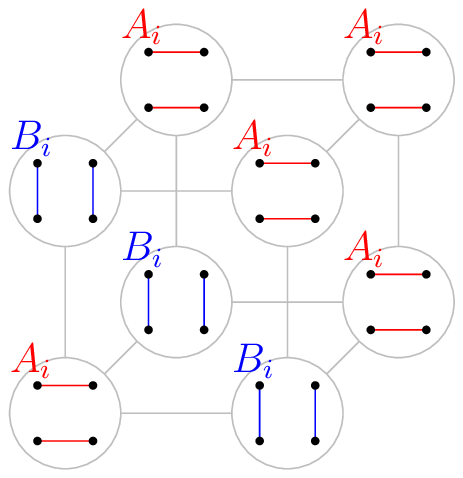}
        \caption{$M_i$}
    \end{subfigure}   
    \quad
    \begin{subfigure}[b]{0.28\textwidth}
        \includegraphics[width=\textwidth]{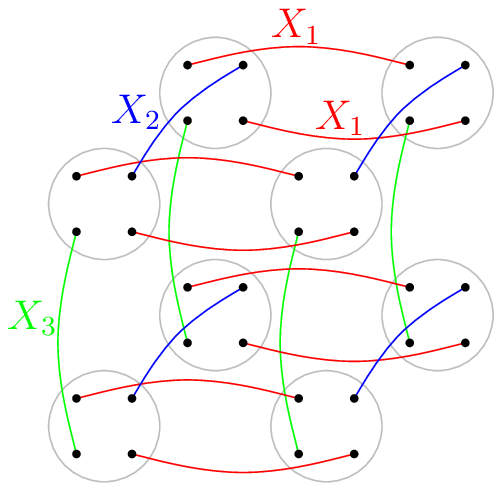}
        \caption{$N_1$}
    \end{subfigure}
    \qquad
    \begin{subfigure}[b]{0.28\textwidth}
        \includegraphics[width=\textwidth]{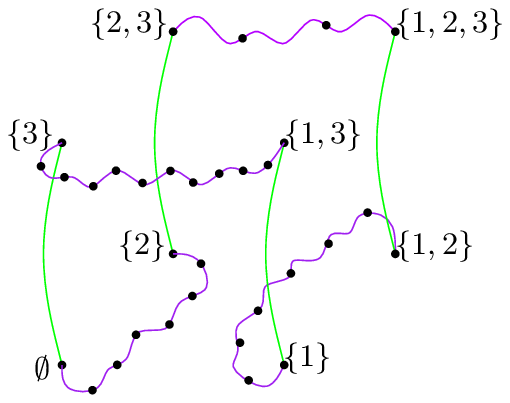}
		\caption{Sketch of cycle ${M_i\cup N_1}$}
		\label{subfig:sketch}
    \end{subfigure}
    \caption{An example when $l=2$}\label{fig:l=3}
\end{figure}

Now $\{M_i\}_{i=1}^l \cup \{N_j\}_{j=1}^3$ is a set of $3+l$ disjoint perfect matchings. It remains to show that $M_i \cup N_j$ is a Hamilton cycle for any $i$ and $j$.

Note that $\{M_i\}$ is invariant under the permutation that cycles directions 1,2 and 3. Since $N_2$ and $N_3$ are obtained from $N_1$ by such cyclic permutations, we can without loss of generality assume that $j=1$.

Consider $M_i \cup N_1$ with the edges in $Q_3^{\emptyset}$ removed; that is, the edges $\emptyset \{3\}$, $\{1\}\{1,3\}$, $\{2\}\{2,3\}$ and $\{1,2\}\{1,2,3\}$.
We will show that the resulting graph comprises four paths, from $\emptyset$ to $\{2\}$, from $\{2,3\}$ to $\{1,2,3\}$, from $\{1,2\}$ to $\{1\}$ and from $\{1,3\}$ to $\{3\}$. Thus when we add back the four edges in direction 3, we get a Hamilton cycle. See figure \ref{subfig:sketch} for an example.

View $Q_{3+l}$ as an $l$-dimensional hypercube whose `vertices' are copies of $Q_3$. Starting at a vertex in $Q_3^\emptyset$ and following the path from it, we will not return to $Q_3^\emptyset$ until we have made $2^l$ steps around $A_i \cup B_i$.

A path starting at $\emptyset$ will move in directions according to $A_i$ then $X_1$ then $B_i$ then $X_2$ and then repeat this pattern. It will return to $Q_3^\emptyset$ after $2^l$ moves from $A_i \cup B_i$ and $2^l- 1$ moves from $X_1 \cup X_2$. Since $l \ge 2$, this means we end at the vertex $\{2\}$, and the path contains $2(2^l)$ vertices.

The same argument works to show that there is a path from $\{1,2\}$ to $\{1\}$ containing $2(2^l)$ vertices.

A path starting at $\{2,3\}$ will move in directions according to $A_i$ then $X_1$ then $A_i$, ending at the vertex $\{1,2,3\}$ and containing $4$ vertices.

A path starting at $\{1,3\}$ will move in directions according to $A_i,  X_1, B_i, X_2, A_i, X_1, A_i, X_2$, and then repeat this pattern. 
It will return to $Q_3^\emptyset$ after $2(2^l)-2$ moves from $A_i \cup B_i$ and $2(2^l)- 3$ moves from $X_1 \cup X_2$. Thus we end at the vertex $\{3\}$, and the path contains $4(2^l)-4$ vertices.

The sum of the lengths of these paths is $8(2^l)$, and so every vertex is contained in one of these paths.
\end{proof}

\begin{prop} There is a 1-factorization $\mathcal{M}$ of the hypercube $Q_{4}$ such that $\GM$ is isomorphic to the complete bipartite graph $K_{3,1}$.
\label{case:k=3,l=1}
\end{prop}
\begin{proof}
The four matchings are shown in figure \ref{fig:3-1}. It is easy to check that the top matching forms a Hamilton cycle with any of the three bottom matchings (in fact, by symmetry you need only check one pair).
\begin{figure}[h!]
    \centering
    \includegraphics[width=.8\textwidth]{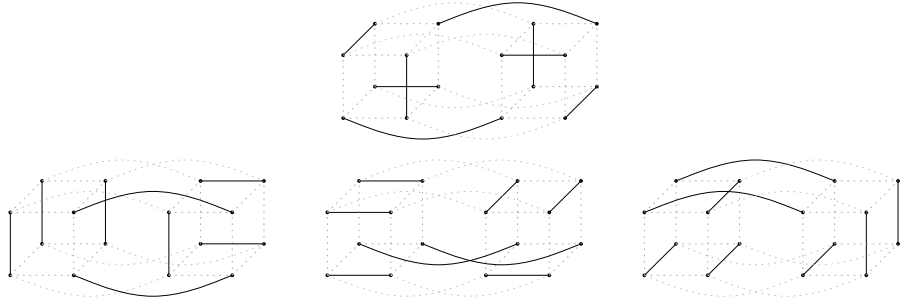}
    \caption{The matchings for $k=1$ and $l=3$}\label{fig:3-1}
\end{figure}
\end{proof}

\begin{proof}[Proof of Theorem \ref{thm:main}]
Combine the results of Propositions \ref{case:neither3}, \ref{case:k=3} and \ref{case:k=3,l=1}.
\end{proof}

\section{Direction Respecting 1-Factorizations}
\label{sec:direction}

The only case not covered by Theorem \ref{thm:main} is whether $\GM$ can be isomorphic to $K_{3,3}$. This case cannot be resolved with our methods alone. To explain why, we will introduce a notion of direction respecting 1-factorizations. 

Fix $k$ and $l$ and let $\mathcal{M} = M_1,M_2,\ldots M_{k+l}$ be a 1-factorization of $Q_{k+l}$. We call the 1-factorization $\mathcal{M}$ \emph{direction respecting} if $M_1,M_2,\ldots, M_k$ only use edges in directions $1,\ldots,k$
 and $M_{k+1},M_{k+2},\ldots M_{k+l}$ only use edges in directions $k+1,\ldots k+l$.
 
 Note that the matchings constructed in Propositions \ref{case:neither3} and \ref{case:k=3} were direction respecting for the appropriate $k$ and $l$. However, the 1-factorisation given in the proof of proposition \ref{case:k=3,l=1} was not direction respecting. We shall prove that there is no direction respecting 1-factorization $\mathcal{M}$ with $\GM$ isomorphic to $K_{3,3}$ or $K_{3,1}$.

\begin{theorem} \label{thm:why3hard}
There is a direction respecting 1-factorization $\mathcal{M}$ of $Q_{k+l}$ with $\GM = K_{k,l}$ if and only if $k,l$ are not $3,1$ or $3,3$. 
\end{theorem}
\begin{proof}
First note that the proof of Theorem \ref{thm:main} shows that such a 1-factorization exists when $k,l$ are not $1,3$ or $3,3$.

Let $d = 3+l$ where $l$ is 3 or 1.
For $M$ a perfect matching on $Q_d$, think of $M$ as a bijection from the odd vertices of $Q_d$ to the even vertices (as in Theorem \ref{thm:Bipartite}). If $M$ and $N$ are perfect matchings then $M N^{-1}$ is a permutation on the even vertices of $Q_d$.
We define the \emph{sign} of a 1-factorization $\{M_i\}_{i=1}^{d}$ of $Q_{d}$ to be the product of the signs of the permutations $M_i M_j^{-1}$ for all $i < j$. That is,
$$\sign(\mathcal{M}) = \prod_{i<j} \sign\left(M_i M_j^{-1}\right).$$

Let $\mathcal{D}^{(d)} = \{D^{(d)}_i\}_{i=1}^{n}$ be the directional matchings of $Q_d$, where $D^{(d)}_i$ contains all edges in direction $i$. For $i \ne j$ the permutation $D^{(d)}_i \left(D^{(d)}_j\right)^{-1}$ consists of $2^{d-2}$ disjoint 4-cycles. Thus $\sign\left(D^{(d)}_i \left(D^{(d)}_j \right)^{-1}\right) = (-1)^{2^{d-2}} = 1$ for all $i,j$, and so $\sign(\mathcal{D}^{(d)})=1$.

Suppose
$\mathcal{M} = \{M_i\}_{i=1}^3\cup\{N_j\}_{j=1}^l$ 
is a 1-factorization of $Q_d$ where $M_i \cup N_j$ is a Hamilton cycle for all $i,j$. The permutation $M_i N_j^{-1}$ is a cycle of length $2^{d-1}$ and so has sign $-1$. Note that $\sign\left(M_i M_s^{-1}\right) = \sign\left(M_i N_j^{-1}\right)\left((M_s N_j^{-1})^{-1}\right) = (-1)(-1) = 1$, and similarly $\sign(N_j N_t^{-1}) = 1$.
Thus $\sign(\mathcal{M}) = (-1)^{3l} = -1$ for any $\mathcal{M}$ with $\GM = K_{3,l}$. 

We will define a switching operation on 1-factorizations that preserves their sign. We will further show that any direction respecting 1-factorization $\mathcal{M}$ can be obtained from $\mathcal{D}^{(d)}$ using a series of switches. Since the sign of $\mathcal{D}^{(d)}$ is $1$, this is enough to show that $\GM \ne K_{3,l}$.

Let $\mathcal{M} = \{M_i\}_{i=1}^d$ be a 1-factorization of $Q_d$. Take a 4-cycle $x,y,v,w$ in $Q_{d}$ and suppose that the edges $xy$ and $vw$ are in matching $M_s$ and $vy$ and $xw$ are in matching $M_t$. A switch on $w,v,y,w$ replaces $\mathcal{M}$ by the 1-factorization $\mathcal{M'} = \{M'_i\}_{i=1}^n$ where 
 $M'_s = M_s  \cup \{ vy, xw\} \setminus \{ xy, vw\}$, $M_t' = M_t \cup \{ xy, vw\} \setminus \{vy, xw\}$, and $M'_i = M_i$ for $i \ne s,t$. 
 
Viewing the 1-factors $M_i$ as bijections from the even vertices to the odd vertices, we have composed $M_s$ and $M_t$ with the function swapping $x$ and $v$, where $x$ and $v$ are the even vertices of $x,y,v,w$. Therefore the permutations $M_s M_i^{-1}$ and $M'_s M_i^{-1}$, where $i \ne s,t$, differ from each other by the transposition $(a,c)$ and so have opposite sign. Similarly $M_t M_i^{-1}$ and $M'_t M_i^{-1}$ have opposite sign.

From this second interpretation of the switch it is clear that:
\begin{align*}
\sign(\mathcal{M'}) &= \prod_{i<j} \sign\left(M'_i (M'_j)^{-1}\right))
\\
&= \prod_{\substack{\text{exactly one of} \\ \text{$i,j$ is $s$ or $t$}}} - \sign\left(M_i M_j^{-1}\right) 
\prod_{\substack{\text{neither or both of} \\ \text{$i,j$ is $s$ or $t$}}} \sign\left(M_i M_j^{-1}\right) 
\\
&= (-1)^{2(d-2)} \sign(\mathcal{M}) =  \sign(\mathcal{M})
\end{align*}

All that is left to show is that a 1-factorization satisfying the conditions of the theorem can be obtained from $\mathcal{D}^{(d)}$ by a series of switches. We will use the following claim.

\begin{claim} Let $D^{(3)}_1,D^{(3)}_2,D^{(3)}_3$ be the directional matchings on $Q_3$ and let $A_1,A_2,A_3$ be another 1-factorization of $Q_3$. 
Then there are a series of switches that transform $D^{(3)}_1$, $D^{(3)}_2$, $D^{(3)}_3$ into $A_1,A_2,A_3$, respecting the ordering.
\end{claim}
\begin{proof}[Proof of Claim]
It is easy to check that there are only 4 ways to partition $Q_3$ into perfect matchings, up to ordering -- one way uses three directional matchings and the other three ways each use one directional matching. Without loss of generality say that $A_1$ is a directional matching.

Note that we can use switches to re-order $D^{(3)}_1,D^{(3)}_2,D^{(3)}_3$. 
To swap $D^{(3)}_i$ and $D^{(3)}_j$ switch on $\emptyset, \{i\}, \{j\}, \{i,j\}$
and on $\{k\}, \{i,k\}, \{j,k\}, \{i,j,k\}$, where $i,j,k$ is $1,2,3$ in some order.
Thus we can assume without loss of generality that $A_1 = D^{(3)}_1$.

If $A_2$ and $A_3$ are also directional matchings then we are done. If not, then we can switch on $\emptyset, \{2\},\{2,3\},\{3\}$ to make them both directional matchings. 
\end{proof}

Let $\mathcal{M} = \{M_i\}_{i=1}^3\cup\{N_j\}_{j=1}^l$ 
be a 1-factorization of $Q_d$ satisfying the conditions of the theorem. 

As in Theorem \ref{thm:main}, we can view $Q_{3+l}$ as an $l$ dimensional hypercube whose `vertices' are copies of $Q_3$. 
For $v \subset \{3+1,\ldots, 3+l\}$ 
let $Q_3^v$ be the induced subgraph of $Q_{3+l}$ on vertices of the form $u \cup v$ for all 
$u \subset \{1,2,3\}$.
For each $v$ in turn, apply the claim to $Q_3^v$ and  $M_1,M_2,M_3$ restricted to $Q_3^v$. In this way we obtain a series of switches that turns $D^{(d)}_1,D^{(d)}_2,D^{(d)}_3$ into $M_1,M_2,M_3$.

If $l= 1$, $N_1 = D^{(n)}_4$ and we are done. If $l = 3$, apply an analagous process to above to find switches that turn $D^{(d)}_4,D^{(d)}_5,D^{(d)}_6$ into $N_1,N_2,N_3$. Note that these switches will be only on edges in directions 4,5,6 and so will not interfere with $M_1,M_2,M_3$ in any way.
\end{proof}

\section{Open Questions}
\label{sec:questions}
The most obvious question is the missing case from Theorem \ref{thm:main}.
\begin{question}
Is it possible to find a 1-factorization $\mathcal{M}$ of $Q_6$ such that $\GM = K_{3,3}$? 
\end{question}
Theorem \ref{thm:why3hard} and its proof show that any such matching $\mathcal{M}$ cannot be obtained from applying a series of switches to the directional matchings. However, there is an example where $\GM = K_{3,1}$, and computer checking suggests that in 4 or 5 dimensions there are many other ways to 1-factorize $Q_d$ and get complete bipartite graphs than the way shown in this paper.

We know from Theorem \ref{thm:Bipartite} that we cannot have a perfect 1-factorization of $Q_d$ for $d>2$. In fact, the maximum possible number of pairs of 1-factors whose union forms a Hamilton cycle is $\left\lfloor \frac{d^2}{4} \right\rfloor$, obtained when $\GM = K_{\lfloor d/2 \rfloor,\lceil d/2 \rceil}$. What can be said about the other pairs -- can their union be close to a Hamilton cycle in some way?
\begin{question}
Let $\mathcal{M} = \{M_i\}_{i=1}^d$ be a 1-factorization of $Q_d$. 
Is it possible for $M_i \cup M_j$ to contain a cycle of length $(1-o(1))2^d$ for every $i\ne j$?
\end{question}
\begin{question}
Let $\mathcal{M} = \{M_i\}_{i=1}^d$ be a 1-factorization of $Q_d$. Is it possible for $M_i \cup M_j$ to consist of at most 2 cycles for every $i\ne j$?
\end{question}
Computer checking shows that for $n \le 5$ the answer to the latter question is `yes' and in dimensions 4 and 5 there are actually several different 1-factorizations that work. 

One could also phrase more general versions of these questions in terms of finding bounds on an appropriate minimax or maximin function. For example,
\begin{question}
For a 1-factorization $\mathcal{M} = \{M_i\}_{i=1}^d$  of $Q_d$, 
let $c_{i,j}$ be the length of the longest cycle in $M_i \cup M_j$ 
and let $f(\mathcal{M}) = min_{i\ne j}(c_{i,j})$. 
Can one find bounds on $max_{\mathcal{M}}(f(\mathcal{M}))$ in terms of $d$?
\end{question}
We can prove that $max_{\mathcal{M}}(f(\mathcal{M}))$ is non-decreasing with $d$. We suspect that it grows exponentially in $d$, but we cannot yet prove it is even better than constant. 

A different way of thinking of Hamilton cycles is as connected 2-factors. Thus a different generalisation of the problem would be to ask about the connectivity of other $r$-factors. For example,

\begin{question} For each $d$, let $r=r(d)$ be minimal subject to there existing a 1-factorization $\mathcal{M}$ of $Q_d$ where the union of any $r$ distinct 1-factors is connected. What is the value of $r(d)$?
\end{question}

Theorem \ref{thm:Bipartite} shows that $r(d)$ is greater than $2$ for $d>2$. 
The 1-factorization given by Theorem \ref{thm:main} in the case $k = \left\lfloor \frac{d}{2} \right\rfloor$ and $l = \left\lceil \frac{d}{2} \right\rceil$ has the property that the union of any $\left( \left\lceil \frac{d}{2} \right\rceil + 1 \right)$ 1-factors is connected, hence $r(d) \le \left\lceil \frac{d}{2} \right\rceil + 1$ for $d \ne 6$. It seems possible that $r$ is constant and it could be even as small as 3.

\nocite{*}
\bibliographystyle{alpha}
\bibliography{factorising}{}

\begin{thebibliography}{BMW02}

\bibitem[And73]{Anderson}
B.~A. Anderson.
\newblock Finite topologies and {H}amiltonian paths.
\newblock {\em J. Combinatorial Theory Ser. B}, 14:87--93, 1973.

\bibitem[Arc95]{Craft}
Dan Archdeacon.
\newblock Problems in topological graph theory.
\newblock
  \url{http://www.cems.uvm.edu/TopologicalGraphTheoryProblems/perfectq.htm},
  1995.
\newblock [Online; accessed 27-July-2018].

\bibitem[BMW02]{Bryant}
Darryn Bryant, Barbara~M. Maenhaut, and Ian~M. Wanless.
\newblock A family of perfect factorisations of complete bipartite graphs.
\newblock {\em J. Combin. Theory Ser. A}, 98(2):328--342, 2002.

\bibitem[CM13]{Chitra}
Vaithiyalingam Chitra and Appu Muthusamy.
\newblock A note on semi-perfect 1-factorization and {C}raft's conjecture.
\newblock {\em Graph Theory Notes N. Y.}, 64:58--62, 2013.

\bibitem[GG10]{Gochev}
Vasil~S. Gochev and Ivan~S. Gotchev.
\newblock On {$k$}-semiperfect 1-factorizations of {$Q_n$} and {C}raft's
  conjecture.
\newblock {\em Graph Theory Notes N. Y.}, 58:36--41, 2010.

\bibitem[KK05]{Kralovic}
Rastislav Kr\'alovi\v{c} and Richard Kr\'alovi\v{c}.
\newblock On semi-perfect 1-factorizations.
\newblock In {\em Structural information and communication complexity}, volume
  3499 of {\em Lecture Notes in Comput. Sci.}, pages 216--230. Springer,
  Berlin, 2005.

\bibitem[KL78]{KotzigLabelle}
Anton Kotzig and Jacques Labelle.
\newblock Strongly {H}amiltonian graphs.
\newblock {\em Utilitas Math.}, 14:99--116, 1978.

\bibitem[Kot64]{Kotzig}
A.~Kotzig.
\newblock Hamilton graphs and {H}amilton circuits.
\newblock In {\em Theory of {G}raphs and its {A}pplications ({P}roc. {S}ympos.
  {S}molenice, 1963)}, pages 63--82. Publ. House Czechoslovak Acad. Sci.,
  Prague, 1964.

\bibitem[Lau80]{Laufer}
P.~J. Laufer.
\newblock On strongly {H}amiltonian complete bipartite graphs.
\newblock {\em Ars Combin.}, 9:43--46, 1980.

\bibitem[Nak75]{Nakamura}
G.~Nakamura.
\newblock Dudeney's round table problem for the cases of $n=p+1$ and $n=2p$.
\newblock {\em Sugaku Sem.}, 159:24--29, 1975.
\newblock [in Japanese].

\bibitem[Sto06]{Stong}
Richard Stong.
\newblock Hamilton decompositions of directed cubes and products.
\newblock {\em Discrete Math.}, 306(18):2186--2204, 2006.

\bibitem[Wal97]{Wallis}
W.~D. Wallis.
\newblock {\em One-factorizations}, volume 390 of {\em Mathematics and its
  Applications}.
\newblock Kluwer Academic Publishers Group, Dordrecht, 1997.

\end{thebibliography}

\end{document}